\documentclass[12pt]{amsproc}
\usepackage{hyperref}
\usepackage{geometry}
\usepackage{graphicx}
\usepackage{amsthm}
\usepackage{amssymb}

\theoremstyle{plain}
\newtheorem{thm}{Theorem}[section]
\newtheorem{cor}[thm]{Corollary}
\newtheorem{lem}[thm]{Lemma}

\newtheorem{defn}[thm]{Definition}

\newtheorem{question}[thm]{Question}

\begin{document}

\title{On Graded $\phi-$Prime Submodules}

\author{Azzh Saad \textsc{Alshehry}}
\address{Department of Mathematical Sciences, Faculty of Sciences, Princess Nourah Bint Abdulrahman University, P.O. Box 84428, Riyadh 11671, Saudi Arabia.}
\email{asalshihry@pnu.edu.sa}

\author{Malik \textsc{Bataineh}}
\address{Department of Mathematics and Statistics, Jordan University of Science and Technology, Irbid, Jordan}
\email{msbataineh@just.edu.jo}

\author{Rashid \textsc{Abu-Dawwas}}
\address{Department of Mathematics, Yarmouk University, Irbid, Jordan}
\email{rrashid@yu.edu.jo}

\subjclass[2010]{Primary 16W50; Secondary 13A02}

\keywords{Graded prime submodules; graded $\phi-$prime submodules.}

\begin{abstract}
Let $R$ be a graded commutative ring with non-zero unity $1$ and $M$ be a graded unitary $R$-module. Let $GS(M)$ be the set of all graded $R$-submodules of $M$ and $\phi: GS(M)\rightarrow GS(M)\bigcup\{\emptyset\}$ be a function. A proper graded $R$-submodule $K$ of $M$ is said to be a graded $\phi-$prime $R$-submodule of $M$ if whenever $r$ is a homogeneous element of $R$ and $m$ is a homogeneous element of $M$ such that $rm\in K-\phi(K)$, then either $m\in K$ or $r\in (K:_{R}M)$. If $\phi(K)=\emptyset$ for all $K\in GS(M)$, then a graded $\phi-$prime submodule is exactly a graded prime submodule. If $\phi(K)=\{0\}$ for all $K\in GS(M)$, then a graded $\phi-$prime submodule is exactly a graded weakly prime submodule. Several properties of graded $\phi-$prime submodules have been investigated.
\end{abstract}

\maketitle

\section{Introduction}

Throughout this article, $G$ will be a group with identity $e$ and $R$ a commutative ring with nonzero unity $1$. Then $R$ is said to be $G$-graded if $R=\displaystyle\bigoplus_{g\in G} R_{g}$ with $R_{g}R_{h}\subseteq R_{gh}$ for all $g, h\in G$ where $R_{g}$ is an additive subgroup of $R$ for all $g\in G$. The elements of $R_{g}$ are called homogeneous of degree $g$. If $x\in R$, then $x$ can be written uniquely as $\displaystyle\sum_{g\in G}x_{g}$, where $x_{g}$ is the component of $x$ in $R_{g}$. It is known that $R_e$ is a subring of $R$ and $1\in R_e$. The set of all homogeneous elements of $R$ is $h(R)=\displaystyle\bigcup_{g\in G}R_{g}$. Assume that $M$ is a left unitary $R$-module. Then $M$ is said to be $G$-graded if $M=\displaystyle\bigoplus_{g\in G}M_{g}$ with $R_{g}M_{h}\subseteq M_{gh}$ for all $g,h\in G$ where $M_{g}$ is an additive subgroup of $M$ for all $g\in G$. The elements of $M_{g}$ are called homogeneous of degree $g$. It is clear that $M_{g}$ is an $R_{e}$-submodule of $M$ for all $g\in G$. If $x\in M$, then $x$ can be written uniquely as $\displaystyle\sum_{g\in G}x_{g}$, where $x_{g}$ is the component of $x$ in $M_{g}$. The set of all homogeneous elements of $M$ is $h(M)=\displaystyle\bigcup_{g\in G}M_{g}$. Let $K$ be an $R$-submodule of a graded $R$-module $M$. Then $K$ is said to be graded $R$-submodule if $K=\displaystyle\bigoplus_{g\in G}(K\cap M_{g})$, i.e., for $x\in K$, $x=\displaystyle\sum_{g\in G}x_{g}$ where $x_{g}\in K$ for all $g\in G$. An $R$-submodule of a graded $R$-module need not be graded. For more details and terminology, see \cite{Hazart, Nastasescue}.

\begin{lem}\label{1}(\cite{Farzalipour}, Lemma 2.1) Let $R$ be a graded ring and $M$ be a graded $R$-module.

\begin{enumerate}

\item If $I$ and $J$ are graded ideals of $R$, then $I+J$ and $I\bigcap J$ are graded ideals of $R$.

\item If $N$ and $K$ are graded $R$-submodules of $M$, then $N+K$ and $N\bigcap K$ are graded $R$-submodules of $M$.

\item If $N$ is a graded $R$-submodule of $M$, $r\in h(R)$, $x\in h(M)$ and $I$ is a graded ideal of $R$, then $Rx$, $IN$ and $rN$ are graded $R$-submodules of $M$. Moreover, $(N:_{R}M)=\left\{r\in R:rM\subseteq N\right\}$ is a graded ideal of $R$.
\end{enumerate}
\end{lem}

Let $I$ be a proper graded ideal of $R$. Then the graded radical of $I$ is $Grad(I)$, and is defined to be the set of all $r\in R$ such that for each $g\in G$, there exists a positive integer $n_{g}$ for which $r_{g}^{n_{g}}\in I$. One can see that if $r\in h(R)$, then $r\in Grad(I)$ if and only if $r^{n}\in I$ for some positive integer $n$. In fact, $Grad(I)$ is a graded ideal of $R$, see \cite{Refai Hailat}. A graded $R$-submodule $K$ of $M$ is called a graded radical $R$-submodule of $M$ if $Grad((K :_{R} M)) = (K :_{R} M)$.

Graded prime ideals play a fundamental role in graded ring theory. One of the natural generalizations of graded prime ideals which have attracted the interest
of several authors in the last two decades is the concept of graded prime submodules, see for example \cite{Dawwas Zoubi Bataineh, Atani, Farzalipour, Oral}. A proper graded $R$-submodule $K$ of $M$ is said to be a graded prime $R$-submodule of $M$ if whenever $r\in h(R)$ and $m\in h(M)$ such that $rm\in K$, then either $m\in K$ or $r\in (K:_{R}M)$. It is known that if $K$ is a graded prime $R$-submodule of $M$, then $(K:_{R}M)$ is a graded prime ideal of $R$. A proper graded $R$-submodule $K$ of $M$ is said to be a graded weakly prime $R$-submodule of $M$ if whenever $r\in h(R)$ and $m\in h(M)$ such that $0\neq rm\in K$, then either $m\in K$ or $r\in (K:_{R}M)$. A proper graded $R$-submodule $K$ of $M$ is said to be a graded almost prime $R$-submodule of $M$ if whenever $r\in h(R)$ and $m\in h(M)$ such that $rm\in K-(K:_{R}M)K$, then either $m\in K$ or $r\in (K:_{R}M)$. So, every graded prime submodule is a graded weakly prime submodule, and every graded weakly prime submodule is a graded almost prime submodule. Let $GS(M)$ be the set of all graded $R$-submodules of $M$ and $\phi: GS(M)\rightarrow GS(M)\bigcup\{\emptyset\}$ be a function. A proper graded $R$-submodule $K$ of $M$ is said to be a graded $\phi-$prime $R$-submodule of $M$ if whenever $r\in h(R)$ and $m\in h(M)$ such that $rm\in K-\phi(K)$, then either $m\in K$ or $r\in (K:_{R}M)$. Since $K-\phi(K) = K-(K\bigcap\phi(P))$, so without loss of generality, throughout this article, we will consider $\phi(K)\subseteq K$. Let $M$ be a $G$-graded $R$-module, $g\in G$ and $K$ be a graded $R$-submodule of $M$ such that $K_{g}\neq M_{g}$. Then $K$ is said to be a $g-\phi-$prime $R$-submodule of $M$ if whenever $r\in R_{e}$ and $m\in M_{g}$ such that $rm\in K-\phi(K)$, then either $m\in K$ or $r\in (K:_{R}M)$. Throughout this article, we use the following functions:

\begin{center}
$\phi_{\emptyset}(K)=\emptyset$ for all $K\in GS(M)$,

\vspace{0.25cm}

$\phi_{0}(K)=\{0\}$ for all $K\in GS(M)$,

\vspace{0.25cm}

$\phi_{1}(K)=(K:_{R}M)K$ for all $K\in GS(M)$,

\vspace{0.25cm}

$\phi_{2}(K)=(K:_{R}M)^{2}K$ for all $K\in GS(M)$ and

\vspace{0.25cm}

$\phi_{w}(K)=\displaystyle\bigcap_{i=1}^{\infty}(K:_{R}M)^{i}K$ for all $K\in GS(M)$.
\end{center}

Clearly, graded $\phi_{\emptyset}-$prime, $\phi_{0}-$prime and $\phi_{1}-$prime submodules are graded prime, graded weakly prime and graded almost prime submodules respectively. Obviously, for any graded submodule and every positive integer $k$, we have the following implications:

\textbf{graded prime submodule $\Rightarrow$ graded $\phi_{w}-$prime submodule $\Rightarrow$ graded $\phi_{k}-$prime submodule $\Rightarrow$ graded $\phi_{k-1}-$prime submodule.}

\vspace{0.2cm}

For functions $\phi, \varphi:GS(M)\rightarrow GS(M)\bigcup\{\emptyset\}$, we write $\phi\leq\varphi$ if $\phi(K)\subseteq\varphi(K)$ for all $K\in GS(M)$. So, if $\phi\leq\varphi$, then every graded $\phi-$prime submodule is graded $\varphi-$prime.

Among several results, we proved that if $K$ is a $g-\phi-$prime $R$-submodule of $M$ such that $(K :_{R_{e}} M)K_{g}\nsubseteq\phi(K)$, then $K$ is a $g$-prime $R$-submodule of $M$ (Theorem \ref{Theorem 2.1}). We showed that if $K$ is a $g-\phi-$prime $R$-submodule of $M$ which is not $g$-prime and $\phi(K)\subseteq(K :_{R_{e}} M)^{2}K_{g}$, then $\phi(K)=(K :_{R_{e}} M)^{i}K_{g}$ for all $i\geq1$ (Corollary \ref{Corollary 2.3}). We proved that if $0\neq m\in M_{g}$ such that $R_{e}m\neq M_{g}$, $(0 :_{R} m) =\{0\}$ and $Rm$ is not a $g$-prime $R$-submodule of $M$, then $Rm$ is not a $g-\phi_{1}-$prime $R$-submodule of $M$ (Theorem \ref{Theorem 2.7}). We showed that if $r\in R_{e}$ such that $rM_{g}\neq M_{g}$ and $(0 :_{M} r)\subseteq rM$, then $rM$ is a $g-\phi_{1}-$prime $R$-submodule of $M$ if and only if it is a $g$-prime $R$-submodule of $M$ (Theorem \ref{Theorem 2.10}). In Theorem \ref{Theorem 2.11}, we introduced a characterization for graded $g-\phi-$prime $R$-submodules. In Theorem \ref{Theorem 2.12}, we study graded $\phi-$prime $R$-submodules over graded quotient $R$-modules. In Theorem \ref{Theorem 2.12 (2)}, we examine graded $\phi-$prime $R$-submodules over multiplicative subsets of $h(R)$. Finally, we proposed an interesting question (Question \ref{2}). Most of the results in this article are inspired from \cite{Zamani}.

\section{Graded $\phi-$Prime Submodules}

In this section, we introduce and study the concept of graded $\phi-$prime submodules.

\begin{defn} Let $M$ be a $G$-graded $R$-module and $\phi: GS(M)\rightarrow GS(M)\bigcup\{\emptyset\}$ be a function.
\begin{enumerate}
\item A proper graded $R$-submodule $K$ of $M$ is said to be a graded $\phi-$prime $R$-submodule of $M$ if whenever $r\in h(R)$ and $m\in h(M)$ such that $rm\in K-\phi(K)$, then either $m\in K$ or $r\in (K:_{R}M)$.

\item Let $K$ be a graded $R$-submodule of $M$ and $g\in G$ such that $K_{g}\neq M_{g}$. Then $K$ is said to be a $g-\phi-$prime $R$-submodule of $M$ if whenever $r\in R_{e}$ and $m\in M_{g}$ such that $rm\in K-\phi(K)$, then either $m\in K$ or $r\in (K:_{R}M)$.
\end{enumerate}
\end{defn}

The concept of $g$-prime ideals have been introduced and investigated in \cite{Dawwas Yildiz}. Let $P$ be a graded ideal of a $G$-graded ring $R$ and $g\in G$ such that $P_{g}\neq R_{g}$. Then $P$ is called a $g$-prime ideal of $R$ if whenever $x, y\in R_{g}$ such that $xy\in P$, then either $x\in P$ or $y\in P$. Motivated by $g$-prime ideals, we introduce the following definitions:

\begin{defn} Let $M$ be a $G$-graded $R$-module, $K$ be a graded $R$-submodule of $M$ and $g\in G$ such that $K_{g}\neq M_{g}$. Then
\begin{enumerate}
\item $K$ is said to a $g$-prime $R$-submodule of $M$ if whenever $r\in R_{e}$ and $m\in M_{g}$ such that $rm\in K$, then either $m\in K$ or $r\in (K:_{R}M)$.

\item $K$ is said to a $g$-weakly prime $R$-submodule of $M$ if whenever $r\in R_{e}$ and $m\in M_{g}$ such that $0\neq rm\in K$, then either $m\in K$ or $r\in (K:_{R}M)$.
\end{enumerate}
\end{defn}

\begin{thm}\label{Theorem 2.1} Let $M$ be a $G$-graded $R$-module, $g\in G$ and $K$ be a $g-\phi-$prime $R$-submodule of $M$. If $(K :_{R_{e}} M)K_{g}\nsubseteq\phi(K)$, then $K$ is a $g$-prime $R$-submodule of $M$.
\end{thm}

\begin{proof}Let $r\in R_{e}$ and $m\in M_{g}$ such that $rm\in K$. If $rm\notin\phi(K)$, then since $K$ is $g-\phi-$prime, we have $r\in(K :_{R} M)$ or $m\in K$. Suppose that $rm\in\phi(K)$. If $rK_{g}\nsubseteq\phi(K)$, then there exists $p\in K_{g}$ such that $rp\notin\phi(K)$, and then $r(m + p)\in K-\phi(K)$. Therefore, $r\in(K :_{R} M)$ or $m + p\in K$, and hence $r\in(K :_{R} M)$ or $m\in K$. So, we may assume that $rK_{g}\subseteq\phi(K)$. If $(K:_{R_{e}}M)m\nsubseteq\phi(K)$, then there exists $u\in(K:_{R_{e}}M)$ such that $um\notin\phi(K)$, and then $(r + u)m\in K-\phi(K)$. Since $K$ is a $g-\phi-$prime submodule, we have $r + u\in(K :_{R} M)$ or $m\in K$. So, $r\in(K :_{R} M)$ or $m\in K$. Therefore, we assume that $(K:_{R_{e}}M)m\subseteq\phi(K)$. Now, since $(K :_{R_{e}} M)K_{g}\nsubseteq\phi(K)$, there exist $s\in(K :_{R_{e}} M)$ and $t\in K_{g}$ such that $st\notin\phi(K)$. So, $(r + s)(m + t)\in K-\phi(K)$, and hence $r + s\in(K :_{R} M)$ or $m + t\in K$. Therefore, $r\in(K :_{R} M)$ or $m\in K$.
\end{proof}

\begin{cor}\label{Corollary 2.2} Let $M$ be a $G$-graded $R$-module, $g\in G$ and $K$ be a $g$-weakly prime $R$-submodule of $M$. If $(K :_{R_{e}} M)K_{g}\neq\{0\}$, then $K$ is a $g$-prime $R$-submodule of $M$.
\end{cor}

\begin{proof} Apply Theorem \ref{Theorem 2.1} by $\phi=\phi_{0}$.
\end{proof}

\begin{cor}\label{Corollary 2.3}Let $M$ be a $G$-graded $R$-module, $g\in G$ and $K$ be a $g-\phi-$prime $R$-submodule of $M$ which is not $g$-prime. If $\phi(K)\subseteq(K :_{R_{e}} M)^{2}K_{g}$, then $\phi(K)=(K :_{R_{e}} M)^{i}K_{g}$ for all $i\geq1$.
\end{cor}

\begin{proof}Since $K$ is not a $g$-prime $R$-submodule of $M$, we have by Theorem \ref{Theorem 2.1}, $(K :_{R_{e}} M)K_{g}\subseteq\phi(K)\subseteq(K :_{R_{e}} M)^{2}K_{g}\subseteq(K :_{R_{e}} M)K_{g}$, which implies that $\phi(K)=(K :_{R_{e}} M)K_{g}=(K :_{R_{e}} M)^{2}K_{g}$. Hence, $\phi(K)=(K :_{R_{e}} M)^{i}K_{g}$ for all $i\geq1$.
\end{proof}

\begin{thm}\label{Theorem 2.7}Let $M$ be a $G$-graded $R$-module, $g\in G$ and $0\neq m\in M_{g}$ such that $R_{e}m\neq M_{g}$ and $(0 :_{R} m) =\{0\}$. If $Rm$ is not a $g$-prime $R$-submodule of $M$, then $Rm$ is not a $g-\phi_{1}-$prime $R$-submodule of $M$.
\end{thm}

\begin{proof}Since $Rm$ is not a $g$-prime $R$-submodule of $M$, there exist $a\in R_{e}$ and $y\in M_{g}$ such that $a\notin(Rm :_{R} M)$, $y\notin Rm$ and $ay\in Rm$. If $ay\notin (Rm :_{R} M)m$, then $Rm$ is not a $g-\phi_{1}-$prime $R$-submodule of $M$. Let $ay\in (Rm :_{R} M)m$. Then $y + m\notin Rm$ and
$a(y + m)\in Rm$. If $a(y + m)\notin (Rm :_{R} M)m$, then $Rm$ is not a $g-\phi_{1}-$prime $R$-submodule of $M$. Let $a(y + m)\in (Rm :_{R} M)m$, then $am\in (Rm :_{R} M)m$, which gives that $am = rm$ for some $r\in (Rm :_{R} M)$. Since $(0 :_{R} m) =\{0\}$, it gives that $a = r\in (Rm :_{R} M)$, which is a contradiction.
\end{proof}

\begin{cor}\label{Corollary 2.8}Let $M$ be a $G$-graded $R$-module, $g\in G$ and $m$ be a non-zero element of $M_{g}$ such that $(0 :_{R} m) =\{0\}$ and $R_{e}m\neq M_{g}$. Then $Rm$ is a $g$-prime $R$-submodule of $M$ if and only if $Rm$ is a $g-\phi_{1}-$prime $R$-submodule of $M$.
\end{cor}

\begin{thm}\label{Theorem 2.10}Let $M$ be a $G$-graded $R$-module, $r\in R_{e}$ and $g\in G$ such that $rM_{g}\neq M_{g}$. If $(0 :_{M} r)\subseteq rM$, then $rM$ is a $g-\phi_{1}-$prime $R$-submodule of $M$ if and only if it is a $g$-prime $R$-submodule of $M$.
\end{thm}

\begin{proof}Suppose that $rM$ is a $g-\phi_{1}-$prime $R$-submodule of $M$. Let $b\in R_{e}$ and $m\in M_{g}$ such that $bm\in rM$. If $bm\notin (rM :_{R} M)rM$, then $b\in (rM :_{R} M)$ or $m\in rM$, since $rM$ is a $g-\phi_{1}-$prime submodule. Suppose that $bm\in (rM :_{R} M)rM$. Now, $(b + r)m\in rM$. If $(b + r)m\notin (rM :_{R} M)rM$, then since $rM$ is a $g-\phi_{1}-$prime submodule, $b + r\in (rM :_{R} M)$ or $m\in rM$, which gives that $b\in (rM :_{R} M)$ or $m\in rM$. Assume that $(b + r)m\in (rM :_{R} M)rM$. Then $bm\in (rM :_{R} M)rM$ gives that $rm\in (rM :_{R} M)rM$. Hence, there exists $y\in (rM :_{R} M)M$ such that $rm = ry$ and so $m-y\in (0 :_{M} r)$. This gives that $m\in (rM :_{R} M)M + (0 :_{M} r)\subseteq rM + (0 :_{M} r)\subseteq rM$. The converse is clear.
\end{proof}

\begin{thm}\label{Theorem 2.11}Let $M$ be a $G$-graded $R$-module, $g\in G$ and $K$ be a graded $R$-submodule of $M$ such that $K_{g}\neq M_{g}$. Then the following statements are equivalent:
\begin{enumerate}
\item $K$ is a $g-\phi-$prime $R$-submodule of $M$.

\item For $m\in M_{g}-K$, $(K :_{R_{e}} m) = (K :_{R_{e}} M)\bigcup(\phi(K) :_{R_{e}} m)$.

\item For $m\in M_{g}-K$, $(K :_{R_{e}} m) = (K :_{R_{e}} M)$ or $(K :_{R_{e}} m) = (\phi(K) :_{R_{e}} m)$.

\item For any ideal $I$ of $R_{e}$ and any graded $R$-submodule $N$ of $M$, if $IN_{g}\subseteq K$ and $IN_{g}\nsubseteq \phi(K)$, then $I\subseteq(K :_{R} M)$ or $N_{g}\subseteq K$.
\end{enumerate}
\end{thm}

\begin{proof}
\begin{enumerate}
$(1)\Rightarrow(2):$ Let $m\in M_{g}-K$ and $r\in (K :_{R_{e}} m)-(\phi(K) :_{R_{e}} m)$. Then $rm\in K-\phi(K)$. Since $K$ is a $g-\phi-$prime $R$-submodule of $M$, $r\in(K :_{R_{e}} M)$. As we may assume that $\phi(K)\subseteq K$, the other inclusion always holds.

$(2)\Rightarrow(3):$ It is known that if a subgroup is the union of two subgroups, then it is equal to one of them.

$(3)\Rightarrow(4):$ Let $I$ be an ideal of $R_{e}$ and $N$ be a graded $R$-submodule of $M$ such that $IN_{g}\subseteq K$. Suppose that $I\nsubseteq(K :_{R} M)$ and $N_{g}\nsubseteq K$. We show that $IN_{g}\subseteq \phi(K)$. Let $r\in I$ and $m\in N_{g}$. If $r\notin(K :_{R_{e}} M)$, then since $rm\in K$, we have $(K :_{R_{e}} m)\neq(K :_{R_{e}} M)$. Hence, by our assumption, $(K :_{R_{e}} m) = (\phi(K) :_{R_{e}} m)$. So, $rm\in \phi(K)$. Now, assume that $r\in I\bigcap(K :_{R_{e}} M)$. Let $u\in I-(K :_{R_{e}} M)$. Then $r + u\in I-(K :_{R_{e}} M)$. So, by the first case, for each $m\in N_{g}$ we have $um\in \phi(K)$ and $(r + u)m\in \phi(K)$. This gives that $rm\in \phi(K)$. Thus in any case $rm\in \phi(K)$. Therefore, $IN_{g}\subseteq \phi(K)$.

$(4)\Rightarrow(1):$ Let $r\in R_{e}$ and $m\in M_{g}$ such that $rm\in K-\phi(K)$. Suppose that $I=R_{e}r$ and $N=Rm$. Then $I$ is an ideal of $R_{e}$ and $N$ is a graded $R$-submodule of $M$ such that $IN_{g}=R_{e}r(Rm)_{g}=R_{e}rR_{e}m_{g}=R_{e}rR_{e}m=R_{e}R_{e}rm\subseteq R_{e}K\subseteq RK\subseteq K$ and $IN_{g}\nsubseteq \phi(K)$. By our assumption, $I\subseteq(K :_{R} M)$ or $N_{g}\subseteq K$, and then $r\in(K :_{R} M)$ or $m\in N\bigcap M_{g}=N_{g}\subseteq K$.
\end{enumerate}
\end{proof}

Let $M$ be a $G$-graded $R$-module and $K$ be a graded $R$-submodule of $M$. Then $M/K$ is a $G$-graded $R$-module by $(M/K)_{g}=(M_{g}+K)/K$ for all $g\in G$. Let $\phi:GS(M)\rightarrow GS(M)\bigcup\{\emptyset\}$ be a function. Define $\phi_{K}:GS(M/K)\rightarrow GS(M/K)\bigcup\{\emptyset\}$ by $\phi_{K}(N/K)=(\phi(N)+K)/K$ for $N\supseteq K$ and $\phi_{K}(N/K)=\emptyset$ for $\phi(N)=\emptyset$.

\begin{thm}\label{Theorem 2.12}Let $M$ be a graded $R$-module and $K$ be a graded $R$-submodule of $M$. If $N$ is a graded $\phi-$prime $R$-submodule of $M$ and $K\subseteq N$, then $N/K$ is a graded $\phi_{K}-$prime $R$-submodule of $M/K$.
\end{thm}

\begin{proof} By (\cite{Saber}, Lemma 3.2), $N/K$ is a graded $R$-submodule of $M/K$. Let $r\in h(R)$ and $m+K\in h(M/K)$ such that $r(m+K)=rm+K\in N/K-\phi_{K}(N/K)$. Then $m\in h(M)$ such that $rm\in N-(\phi(N)+K)$, and then $rm\in N-\phi(N)$. Since $N$ is graded $\phi-$prime, $r\in (N:_{R}M)$ or $m\in N$, and then $r\in (N/K:_{R}M/K)$ or $m\in N$. Therefore, $N/K$ is a graded $\phi_{K}-$prime $R$-submodule of $M/K$.
\end{proof}

Let $M$ be a $G$-graded $R$-module and $S\subseteq h(R)$ be a multiplicative set. Then $S^{-1}M$ is a $G$-graded $S^{-1}R$-module with $(S^{-1}M)_{g}=\left\{\frac{m}{s},m\in M_{h}, s\in S\cap R_{hg^{-1}}\right\}$ for all $g\in G$, and $(S^{-1}R)_{g}=\left\{\frac{a}{s},a\in R_{h}, s\in S\cap R_{hg^{-1}}\right\}$ for all $g\in G$. If $K$ is a graded $R$-submodule of $M$, then $S^{-1}K$ is a graded $S^{-1}R$-submodule of $S^{-1}M$. It is well known that there is a one-to-one correspondence between the set of all graded prime $R$-submodules $K$ of $M$ with $(K :_{R} M)\bigcap S =\emptyset$ and the set of all graded prime $S^{-1}R$-submodules $S^{-1}K$ of $S^{-1}M$. For a graded $R$-submodule $K$ of $M$, $K(S)=\left\{m\in M:\mbox{ there exists }s\in S\mbox{ such that }sm\in K\right\}$ is a graded $R$-submodule of $M$ containing $K$ and $S^{-1}(K(S))=S^{-1}K$. Let $\phi:GS(M)\rightarrow GS(M)\bigcup\{\emptyset\}$ be a function. Define $(S^{-1}\phi):GS(S^{-1}M)\rightarrow GS(S^{-1}M)\bigcup\{\emptyset\}$ by $(S^{-1}\phi)(S^{-1}K)=S^{-1}(\phi(K(S))$ if $\phi(K(S))\neq\emptyset$ and $(S^{-1}\phi)(S^{-1}K)=\emptyset$ if $\phi(K(S))=\emptyset$. Note that, $(S^{-1}\phi_{\emptyset})=\phi_{\emptyset}$ and $(S^{-1}\phi_{0})=\phi_{0}$.

\begin{thm}\label{Theorem 2.12 (2)}Let $M$ be a $G$-graded $R$-module and $K$ be a graded $\phi-$prime $R$-submodule of $M$. Suppose that $S$ is a multiplicative subset of $h(R)$. If $S^{-1}K\neq S^{-1}M$ and $S^{-1}(\phi(K))\subseteq(S^{-1}\phi)(S^{-1}K)$, then $S^{-1}K$ is a graded $(S^{-1}\phi)-$prime $S^{-1}R$-submodule of $S^{-1}M$. Moreover, if $g\in G$ such that $(S^{-1}K)\bigcap M_{g}\neq S^{-1}((\phi(K))$, then $K(S)\bigcap M_{g} = K$.
\end{thm}

\begin{proof} Let $r/s\in h(S^{-1}R)$ and $m/t\in h(S^{-1}M)$ such that $(r/s)(m/t)\in S^{-1}K-(S^{-1}\phi)(S^{-1}K)$. Then $rm/st\in S^{-1}K-S^{-1}(\phi(K))$, and then there exists $u\in S$ such that $urm\in K-\phi(K)$ (note that, for each $v\in S$, $vrm\notin \phi(K)$). Since $K$ is graded $\phi-$prime and $(K :_{R}
M)\bigcap S=\emptyset$, it gives that $rm\in K-\phi(K)$ and so $r\in(K :_{R} M)$ or $m\in K$, and then $r/s\in S^{-1}(K:_{R}M)\subseteq(S^{-1}K:_{S^{-1}R}S^{-1}M)$ or $m/t\in S^{-1}K$. Hence, $S^{-1}K$ is a graded $(S^{-1}\phi)-$prime $S^{-1}R$-submodule of $S^{-1}M$. Moreover, suppose that $(S^{-1}K)\bigcap M_{g}\neq S^{-1}((\phi(K))$ for some $g\in G$ and $m\in K(S)\bigcap M_{g}$. Then there exists $s\in S$ such that $sm\in K$. If $sm\notin \phi(K)$, then $m\in K$. If $sm\in \phi(K)$, then $m\in \phi(K)(S)$. So, $K(S)\bigcap M_{g} = K\bigcup(\phi(K)(S))$. Hence, either $K(S)\bigcap M_{g} = K$ or $K(S)\bigcap M_{g} = (\phi(K)(S))$. If $K(S)\bigcap M_{g} = (\phi(K)(S))$, then we have $(S^{-1}K)\bigcap M_{g} = (S^{-1}K(S))\bigcap M_{g} = S^{-1}(\phi(K)(S)) = S^{-1}(\phi(K))$, which is not the case. So, $K(S)\bigcap M_{g} = K$.
\end{proof}

\begin{question}\label{2}Let $M$ be a graded $R$-module, $S$ be a multiplicative subset of $h(R)$ and $K$ be a graded $R$-submodule of $M$. If $S^{-1}K$ is a graded $S^{-1}\phi-$prime $S^{-1}R$-submodule of $S^{-1}M$, then clearly, $(K :_{R} M)\bigcap S =\emptyset$. In general, under what conditions $K$ will be a graded $\phi-$prime $R$-submodule of $M$?. Even in the cases $\phi= \phi_{0}$, $\phi= \phi_{1}$ and $\phi= \phi_{2}$, we could not answer this question.
\end{question}

\section*{\textbf{Acknowledgement}}

This research was funded by the Deanship of Scientific Research at Princess Nourah bint Abdulrahman University through the Fast-track Research Funding Program.

\end{document}